\documentclass[a4paper, 12pt]{article}
\usepackage[cp1251]{inputenc}
\usepackage[english]{babel}
\usepackage{amsthm}
\usepackage{cite}
\usepackage{tikz-cd}
\usepackage{amsmath}
\usepackage{amsfonts}
\usepackage{amssymb}
\usepackage{hyperref}
\usepackage[]{authblk}
\usepackage{epsfig,graphicx}

\DeclareMathOperator{\Aut}{Aut}
\DeclareMathOperator{\End}{End}

\DeclareMathOperator{\Det}{det}

\DeclareMathOperator{\Deg}{deg}

\DeclareMathOperator{\Sp}{Sp}
\DeclareMathOperator{\Ht}{ht}
\DeclareMathOperator{\J}{J}

\DeclareMathOperator{\GL}{GL}
\newcommand{\TAut}{\operatorname{TAut}}

\newtheorem{thm}{Theorem}[section]
\newtheorem{lem}[thm]{Lemma}

\newtheorem{prop}[thm]{Proposition}

\newtheorem{conj}[thm]{Conjecture}
\newtheorem{Def}[thm]{Definition}

\begin{document}
\renewcommand{\thefootnote}{\fnsymbol{footnote}}
\renewcommand{\thefootnote}{\arabic{footnote}}
\fontsize{12}{12pt}\selectfont
\title{\bf Automorphisms of Weyl Algebra and a Conjecture of Kontsevich}
\renewcommand\Affilfont{\itshape\small}
\author[1,3]{Alexei Kanel-Belov\thanks{kanel@mccme.ru}}
\author[2]{Andrey Elishev\thanks{elishev@phystech.edu}}
\author[1]{Jie-Tai Yu\thanks{jietaiyu@szu.edu.cn}}

\affil[1]{College of Mathematics and Statistics, Shenzhen University, Shenzhen, 518061, China}
\affil[2]{Laboratory of Advanced Combinatorics and Network Applications, Moscow Institute of Physics and Technology, Dolgoprudny, Moscow Region, 141700, Russia}
\affil[3]{Mathematics Department, Bar-Ilan University, Ramat-Gan, 52900, Israel}

\date{}

\maketitle

\renewcommand{\abstractname}{Abstract}
\begin{abstract}
We outline the proof of a conjecture of Kontsevich on the isomorphism between the group of polynomial symplectomorphisms in $2n$ variables and the group of automorphisms of the $n$-th Weyl algebra over complex numbers. Our proof uses lifting of polynomial symplectomorphisms to Weyl algebra automorphisms by means of approximation by tame symplectomorphisms and gauging of the lifted morphism. Approximation by tame symplectomorphisms is the symplectic version of the well-known theorem of D. Anick and is a result of our prior work.

\end{abstract}

\section{Introduction}

The objective of this short note is to explain a method of attack on a conjecture formulated by one of us together with Kontsevich in \cite{BKK1}, which is henceforth referred to as the Kontsevich conjecture. The conjecture, in its most straightforward form, states that the automorphism group of the $n$-th Weyl algebra over an algebraically closed field $\mathbb{K}$ of characteristic zero is isomorphic to the group of polynomial symplectomorphisms -- that is, polynomial automorphisms preserving the symplectic structure -- of the affine space $\mathbb{A}^{2n}_{\mathbb{K}}$. The statement can be reformulated as
\begin{conj}
  $\Aut W_n(\mathbb{K})\simeq \Aut P_n(\mathbb{K})$
\end{conj}
-- where the $n$-th Weyl algebra $W_{n}(\mathbb{K})$ over $\mathbb{K}$ is by definition the quotient of the free associative algebra
\begin{equation*}
\mathbb{K}\langle a_1,\ldots,a_n,b_1,\ldots,b_n\rangle
\end{equation*}
by the two-sided ideal generated by elements
\begin{equation*}
b_ia_j-a_jb_i-\delta_{ij},\;\;a_ia_j-a_ja_i,\;\;b_ib_j-b_jb_i,
\end{equation*}
with $1\leq i,j\leq n$, while the algebra $P_n(\mathbb{K})$ is the commutative polynomial algebra $\mathbb{K}[x_1,\ldots,x_{2n}]$ carrying an additional structure of the Poisson algebra via the standard Poisson bracket -- that is, a bilinear map $$\lbrace \;,\;\rbrace:\mathbb{K}[x_1,\ldots,x_{2n}]\otimes \mathbb{K}[x_1,\ldots,x_{2n}]\rightarrow \mathbb{K}[x_1,\ldots,x_{2n}]$$ that turns $\mathbb{K}[x_1,\ldots,x_{2n}]$ into a Lie algebra and acts as a derivation with respect to polynomial multiplication (therefore, automorphisms of such algebras are required to preserve the additional structure). The standard Poisson bracket is defined as
\begin{equation*}
\lbrace x_i,x_j\rbrace = \delta_{i,n+j}-\delta_{i+n,j},
\end{equation*}
with $\delta_{ij}$ meaning the Kronecker delta.

Several generalizations of Kontsevich conjecture are known; the most obvious one is obtained by replacing the algebraically closed ground field $\mathbb{K}$ with the field $\mathbb{Q}$ of rational numbers. Other generalizations are somewhat more elaborate and are discussed at length in \cite{BKK1}.

The setting in which Conjecture 1 naturally arises is that of deformation quantization of polynomial algebra. The commutative Poisson algebra $P_n$ serves as the classical counterpart to the algebra $W_n$ of polynomial differential operators. It is therefore sensible to ask whether the quantization preserves the automorphism group. One then, in order to answer this question, tries to construct either a direct homomorphism
$$
\Aut W_n(\mathbb{K})\rightarrow \Aut P_n(\mathbb{K})
$$
or an inverse
$$
\Aut P_n(\mathbb{K})\rightarrow \Aut W_n(\mathbb{K}).
$$

Known ways of accomplishing that goal are somewhat involved. We briefly comment on the relatively accessible case of \textbf{(direct) group homomorphism}
$$
\Aut W_n(\mathbb{K})\rightarrow \Aut P_n(\mathbb{K})
$$
with $\mathbb{K}$ being algebraically closed \footnote{One makes a straightforward observation that the Kontsevich conjecture, along with objects accompanying it, are statements which can be formulated by means of first-order logic; therefore, in the case of algebraically closed ground field $\mathbb{K}$ of characteristic zero, one may well work with complex numbers, $\mathbb{K}=\mathbb{C}$, in accordance with the Lefschetz principle.}.

The idea is to realize the Weyl algebra $W_n$ as a subalgebra in an algebra whose center is large enough to contain the polynomial algebra $\mathbb{C}[x_1,\ldots,x_{2n}]$, and then, starting with an automorphism $\varphi\in\Aut W_n(\mathbb{C})$, restrict it to the polynomial algebra. At this point, reduction to positive characteristic starts playing an important role. Namely, one represents the ground field as a reduced direct product of algebraically closed fields of characteristic $p$,
$$
\mathbb{C}\simeq \prod_{p}\; '\; \mathbb{F}_p
$$
where $p$ runs over all prime numbers, and reduction is taken with respect to a free ultrafilter $\mathcal{U}$ on the index set of natural numbers. In other words, $\prod'_{p} \mathbb{F}_p$ is the quotient of the direct product $\prod_{p} \mathbb{F}_p$ by the maximal \footnote{As well as minimal -- recall that the product of fields is always von Neumann regular.} ideal generated by $\mathcal{U}$ as in \cite{FMS}. Such a procedure is sometimes referred to in literature as reduction modulo infinite prime (the infinite prime being the sequence of prime numbers that indexes the direct product -- in this case the standard prime number sequence; one could as well take for such a sequence any sequence of primes which is not equivalent under the chosen ultrafilter to a stationary sequence). The point of this construction consists in the fact that in positive characteristic the Weyl algebra $W_n$ has a huge center isomorphic to the polynomial algebra $\mathbb{F}_p[x_1^p,\ldots,x_n^p,d_1^p,\ldots,d_n^p]$ (with $x_i,\; d_j$ being the generators), while nothing of the sort is the case of characteristic zero. Therefore, the reduced product of Weyl algebras will just be the larger algebra one is looking for -- one whose center contains a copy of the polynomial algebra. A significant detail is given by the fact that the Weyl algebra commutator naturally induces a Poisson structure on the polynomial subalgebra of the larger center, thus making the resulting automorphism symplectic.

Thus one constructs a homomorphism $\Aut W_n(\mathbb{C})\rightarrow \Aut P_n(\mathbb{C})$ which is a candidate for the simplest version of the Kontsevich conjecture. For the sake of brevity we have left out the details and refer the interested reader to the works \cite{BKK1, BKK2} and \cite{K-BE2}. This homomorphism is injective, and induces an isomorphism of subgroups of tame automorphisms (the definition of tame automorphism is given below). An identical procedure produces a monoid homomorphism between the sets of endomorphisms of $W_n$ and $P_n$; this fact has been used to establish a stable equivalence between the Dixmier conjecture \cite{Dix} (any endomorphism of $W_n$ is invertible -- open for all $n$ as of time of writing) and the Jacobian conjecture, cf. \cite{BKK2, Tsu1, Tsu2}.

The direct homomorphism can be made explicitly independent of the prime number sequence by means of a non-standard (inverse) Frobenius twist of coefficients. It is, however, insufficient to guarantee its independence of the choice of infinite prime and the ultrafilter completely, for integer combinations of coefficients (coming from applying Weyl algebra commutation relations) could differ for different such choices. The question of independence of infinite prime is therefore non-trivial. In our prior work \cite{K-BE2} we have provided a proof of independence, which however relies on the homomorphism in question to be one-to-one.

\medskip

The present paper focuses on the reverse approach. Starting with a polynomial symplectomorphism, we construct an automorphism of the power series completion of the Weyl algebra and then argue that the power series that are images of the generators of $W_n$ must be polynomials. The procedure is referred to as the \textbf{lifting} throughout the text. Central to it is the fact that the subgroup of tame symplectomorphisms $\TAut P_n$ is dense (with respect to power series topology) in $\Aut P_n$ -- a fact we have obtained recently \cite{KGE}. Approximation of polynomial automorphisms by tame automorphisms was developed by Anick \cite{An} and has already become a classical result in algebraic geometry. Our work serves, in a way, as a development of Anick's results to the symplectic case.

\section{Tame symplectomorphisms, topology, and approximation}

This section reviews the background and results on approximation by tame automorphisms necessary in our context. Most of the theory, as well as detailed proofs, can be found in the classical work of Anick \cite{An}. Tame symplectomorphism approximation is the main result of our recent work with S. Grigoriev and W. Zhang \cite{KGE}.

Let $A_N=\mathbb{K}[x_1,\ldots, x_{N}]$ be the commutative polynomial algebra over a field $\mathbb{K}$, and let $\varphi$ be an algebra endomorphism.

Any such endomorphism can be identified with the ordered set
\begin{equation*}
(\varphi(x_1),\;\varphi(x_2),\;\ldots,\; \varphi(x_N))
\end{equation*}
of images of generators of the algebra. The polynomials $\varphi(x_i)$ may be represented as sums of their respective homogeneous components; this may be expressed formally as
\begin{equation*}
\varphi = \varphi_0+\varphi_1+\cdots,
\end{equation*}
where $\varphi_k$ is a string of length $N$ whose entries are homogeneous polynomials of total degree $k$.\footnote{We set $\Deg x_i=1$.}
\begin{Def}
The height $\Ht(\varphi)$ of an endomorphism $\varphi$ is defined as
\begin{equation*}
\Ht(\varphi)=\inf\lbrace k\;|\;\varphi_k\neq 0\rbrace,\;\;\Ht(0)=\infty.
\end{equation*}
\end{Def}

This is not to be confused with \textbf{degree} of endomorphism, which is defined as $$\Deg(\varphi)=\sup\lbrace k\;|\;\varphi_k\neq 0\rbrace.$$ The height $\Ht(f)$ of a polynomial $f$ is defined quite similarly to be the minimal number $k$ such that the homogeneous component $f_k$ is non-zero. Evidently, for an endomorphism $\varphi=(\varphi(x_1),\;\ldots,\;\varphi(x_N))$ one has
\begin{equation*}
\Ht(\varphi)=\inf\lbrace \Ht(\varphi(x_i))\;|\;1\leq i\leq N\rbrace.
\end{equation*}

The function
\begin{equation*}
d(\varphi,\psi)=\exp(-\Ht(\varphi-\psi))
\end{equation*}
is a metric on the set $\End \mathbb{K}[x_1,\ldots,x_N]$; the corresponding topology will be referred to as the power series topology.

\medskip

We turn to automorphisms and define the tame subgroup.
\begin{Def}
We say that an automorphism $\varphi \in \Aut A_N$ is \textbf{elementary} if it is given either by a linear change of generators
$$
(x_1,\;\ldots,\;x_N)\mapsto (x_1,\;\ldots,\;x_N)A,\;\;A\in\GL(N,\mathbb{K})
$$
or by a transvection -- a change of variables of the form
$$
(x_1,\;\ldots,\;x_N)\mapsto (x_1,\;\ldots,\;x_k+f(x_1,\ldots,x_{k-1},x_{k+1},\ldots,x_N),\;\ldots,\;x_N)
$$
(that is, all generators are kept fixed with the exception of $x_k$, to which a polynomial free of $x_k$ is added).
\end{Def}

\begin{Def}
The subgroup $\TAut A_N$ of tame automorphisms is the subgroup of $\Aut A_N$ generated by elementary automorphisms defined as above.
\end{Def}

Whenever $N=2n$ is even, the polynomial algebra can be made into $P_n$ by partitioning the set of its generators into two even subsets, $\lbrace x_1,\ldots, x_n\rbrace$ and $\lbrace p_1,\ldots, p_n\rbrace$, and introducing the corresponding Poisson bracket. Under the identification $$\mathbb{K}[x_1,\ldots,x_{n},p_1,\ldots,p_n]\simeq \mathcal{O}(\mathbb{A}^{2n}_{\mathbb{K}})$$ the generators $x_i$, $p_j$ become the Darboux coordinate functions for the standard symplectic form $\omega = \sum_{i} dx_i\wedge dp_i$. The group $\Aut P_n$ is then the subgroup of $\Aut \mathbb{K}[x_1,\ldots,x_{n},p_1,\ldots,p_n]$ of automorphisms which preserve the symplectic (or Poisson) structure. Its intersection with $\TAut \mathbb{K}[x_1,\ldots,x_{n},p_1,\ldots,p_n]$ is the \textbf{subgroup} $\TAut P_n$ \textbf{of tame symplectomorphisms}.

Automorphisms that are not tame are called \textbf{wild}. Wild automorphisms exist -- the most well-known example being due to Nagata \cite{Shes2, Shes3}:
$$
\varphi_N:\mathbb{K}[x,y,z]\rightarrow \mathbb{K}[x,y,z],
$$
\begin{equation*}
\varphi_N:(x,y,z)\mapsto(x+(x^2-yz)x,\;y+2(x^2-yz)x+(x^2-yz)^2z,\;z).
\end{equation*}

In dimension two, all automorphisms are tame -- a fact that allows for an explicit description of $\Aut P_1$ (cf. \cite{Jung, VdK}) and $\Aut W_1$ and, consequently, positive resolution of Kontsevich conjecture in this case. The latter is due to Makar-Limanov \cite{ML1, ML2}. It turns out \cite{ML3} that the tameness of the planar case is not specific to the commutative polynomial algebra, but rather is a property of a broader class of objects.

It is not known whether in even dimensions all symplectomorphisms are tame; that fact, if it were to be established, would pave the way for a quick resolution of Kontsevich conjecture, for the direct homomorphism
$$
\Aut W_n(\mathbb{C})\rightarrow \Aut P_n(\mathbb{C})
$$
restricts to an isomorphism of tame subgroups.

\medskip

We proceed by formulating basic results on approximation by tame automorphisms.

\begin{lem}
An elementary linear automorphism is a symplectomorphism if and only if its matrix $A$ is symplectic, $A\in \Sp(2n, \mathbb{K})$. A transvection defined by a polynomial $f$ is a symplectomorphism if and only if $f$ is free of all generators of the type that has the generator to which $f$ is added. That is, if $f$ is added to $x_k$, then $f$ must be a function of $p_1,\ldots, p_n$ only for it to be a symplectomorphism, and vice versa.
\end{lem}
\begin{proof}
Straightforward.
\end{proof}

We now formulate the basic results of approximation by tame automorphisms. Anick's theorem states that the subgroup $\TAut A_N$ is dense in $\Aut A_N$ in power series topology. The unit Jacobian requirement is not essential (indeed, any automorphism must have a constant Jacobian -- an easy exercise; forcing an automorphism to have unit Jacobian amounts then to a rescaling), yet convenient. One may, without loss of generality, develop approximation for automorphisms in the neighborhood of the identity automorphism -- that is, automorphisms which are identity modulo terms of certain height. In this framework, the unit Jacobian requirement becomes redundant.

\begin{thm}[Anick, \cite{An}]
Let $\varphi=(\varphi(x_1),\;\ldots,\;\varphi(x_N))$ be an automorphism of the polynomial algebra $A_N=\mathbb{K}[x_1,\ldots,x_N]$ over a field $\mathbb{K}$ of characteristic zero, such that its Jacobian
\begin{equation*}
\J(\varphi)=\Det \left[\frac{\partial \varphi(x_i)}{\partial x_j}\right]
\end{equation*}
is equal to $1$. Then there exists a sequence $\lbrace \psi_k\rbrace\subset \TAut \mathbb{K}[x_1,\ldots,x_N]$ of tame automorphisms which converges to $\varphi$ in power series topology.
\end{thm}
The symplectic version of Anick's theorem is a recent development \cite{KGE}.
\begin{thm}
Let $\sigma=(\sigma(x_1),\;\ldots,\;\sigma(x_n),\;\sigma(p_1),\;\ldots,\;\sigma(p_n))$ be a symplectomorphism of $\mathbb{K}[x_1,\ldots,x_n,p_1,\ldots,p_n]$ with unit Jacobian.
Then there exists a sequence $\lbrace \tau_k\rbrace\subset \TAut P_n(\mathbb{K})$ of tame symplectomorphisms which converges to $\sigma$ in power series topology.
\end{thm}
The reader is encouraged to browse the proof of this statement in \cite{KGE} in order to gain a somewhat broader understanding of the context of Kontsevich conjecture and associated situations. 

\smallskip

In order to utilize the approximation theory, we need to be able to make sense of the \textbf{lifted limit} of a tame sequence $\lbrace \sigma_k\rbrace$. Just as the automorphisms $\psi_k$ lifted from $\sigma_k$ are defined by means of formal power series in $\hbar$ in the framework of deformation quantization (see below), so will be the lifted limit $\Psi$. However, while the tame automorphisms $\psi_k$ will have entries polynomial in $\hbar$ (which is an immediate consequence of the quantization formula) and also the coefficients at each $\hbar^n$ will be polynomial in the generators of the Poisson algebra $P_n$, it will generally not be the case for arbitrary $\sigma$. In order for the lifted limit to be well defined, one therefore needs statements on convergence of the appropriate power series: the power series in $\hbar$ with respect to the $\hbar$-adic topology as well as the power series which determine the coefficients (in the $\mathfrak{m}$-adic topology obtained from the scheme structure on $\Aut P_n$). The needed statements translate into the following theorems \cite{KGE}:

\begin{thm}
Let $\sigma$ be a symplectomorphism and let $\mathcal{O}_{\sigma}$ be the local ring of $\Aut P_n(\mathbb{C})$ with its maximal ideal $\mathfrak{m}$. Then there exists a sequence of tame symplectomorphisms $\lbrace \sigma_k\rbrace$ which converges to $\sigma$ in power series topology, such that the coordinates of $\sigma_k$ converge to coordinates of $\sigma$ in $\mathfrak{m}$-adic topology.
\end{thm}  

\begin{thm}
Suppose given a symplectomorphism $\sigma$ and $\lbrace \sigma_k\rbrace$ is a tame sequence converging to it. Then the $\hbar$-series which define the lifted tame automorphisms $\psi_k$ converge to the power series that define the limit $\Psi$ in the $\hbar$-adic topology.
\end{thm} 


\section{Lifting of polynomial symplectomorphisms}

Given an arbitrary symplectomorphism $\sigma \in\Aut P_n(\mathbb{C})$ and a sequence of tame symplectomorphisms $\lbrace \sigma_k\rbrace$ converging to it, we can construct a sequence $\lbrace \psi_k\rbrace$ of Weyl algebra automorphisms in the following way. Let $\psi_k$ be the pre-image of $\sigma_k$ under the direct homomorphism $\Aut W_n\rightarrow \Aut P_n$ described in the first section. It is an isomorphism of the tame subgroups, therefore the assignment is well defined and unique. Alternatively, we could start with the Poisson algebra $P_n=\mathbb{C}[x_1,\ldots,x_n,p_1,\ldots, p_n]$ and perform the deformation quantization \cite{Kon, Kell} according to a formal parameter $\hbar$ and the associative star product $\star$. It is straightforward to deform elementary symplectomorphisms, and the procedure yields, under appropriate identifications, the same result as the one involving the direct homomorphism. In either case, we refer to thus described procedure as the \textbf{lifting} of polynomial symplectomorphisms.

\medskip

It seems slightly more convenient to work with the deformed family $W_n(\hbar)$ of Weyl algebras depending upon $\hbar$ rather than with a single Weyl algebra. The most important thing to bear in mind, however, is the fact that a given symplectomorphism $\sigma$ will specify (by imposition of associativity and Weyl algebra commutation relations on images $\sigma(x_i)$, $\sigma(p_j)$) a new star product $\star_{\sigma}$, which differs from the original one by a gauge transformation and defines a new family of associative algebras $W_n(\hbar, \sigma)$. The main theorem then admits a reformulation in the following way:

\begin{thm}[Main theorem]
There is an algebra embedding
$$
W_n(\hbar, \sigma)\hookrightarrow W_n(\hbar).
$$
\end{thm}

This is equivalent to the Kontsevich conjecture and can therefore be perceived as the principal subject of this study.

\smallskip

It is worth mentioning that Myung and Oh \cite{MyOh} have recently conducted a study of deformations of Poisson algebras with the purpose similar to that of the present paper. It can be inferred from their results that the larger algebras $S_n(\hbar)$ and $S_n(\hbar, \sigma)$ of formal power series (while $W_n(\hbar)$ consists of expressions polynomial in $\hbar$) are isomorphic to each other. The statement of the Kontsevich conjecture (provided by the Main theorem above) is stronger.

\smallskip

Whenever $\sigma$ is tame, the statement of the theorem is straightforward.
Any tame symplectomorphism $\sigma$ lifts to an object of the following form

$$
(\Psi_1(x_1,\ldots,x_n,p_1,\ldots,p_n,\hbar),\ldots,\Psi_{2n}(x_1,\ldots,x_n,p_1,\ldots,p_n,\hbar)).
$$
Here $\Psi_l(x_1,\ldots,x_n,p_1,\ldots,p_n,\hbar)$ are power series in $\hbar$ whose coefficients are polynomials in (commuting) variables $x_i$ and $p_j$. They are obtained by applying, say, the Kontsevich quantization formula \cite{Kon} (rather, a special case corresponding to the $2n$-dimensional affine space) to the polynomials $\sigma(x_1),\;\ldots,\;\sigma(p_n)$. It is clear from the fact that the coefficients at $\hbar^n$ are the images under certain bidifferential operators that the power series $\Psi_l$ are really polynomials in $\hbar$, whose degree (in $\hbar$) depends on the total degree of $\sigma$.

\smallskip

The case of general $\sigma$ cannot be processed in this way. Indeed, if $\lbrace\sigma_k\rbrace$ is a sequence of tame symplectomorphisms converging to $\sigma$, one can take the lifted automorphisms $\psi_k$ and define the lifted limit $\psi$. The lifting procedure based on the initial star product, however, when applied to the sequence $\sigma_k$, will in the limit return an object defined by power series in $\hbar$ (rather than polynomials). Moreover, the coefficients at $\hbar^n$ will also in general be power series in commuting variables $x_i$ and $p_j$ (although that particular problem can be dealt with, cf. lemma in the next section).

At this point (in accordance with the remark at the end of the previous chapter), it may be somewhat comforting to note that these coefficients in $x_i$ and $p_j$ will always be given by power series with sufficiently good behavior. In fact, they will be convergent with respect to the $\mathfrak{m}$-adic topology of the local ring $\mathcal{O}_{\sigma}$ of the scheme $\Aut P_n$ at the point $\sigma$. More precisely, we have the following theorem \cite{KGE}:
\begin{thm}
Let $\sigma$ be a symplectomorphism and let $\mathcal{O}_{\sigma}$ be the local ring of $\Aut P_n(\mathbb{C})$ with its maximal ideal $\mathfrak{m}$. Then there exists a sequence of tame symplectomorphisms $\lbrace \sigma_k\rbrace$ which converges to $\sigma$ in power series topology, such that the coordinates of $\sigma_k$ converge to coordinates of $\sigma$ in $\mathfrak{m}$-adic topology.
\end{thm}  
This result says that the power series which constitute the coefficients at $\hbar^n$ of the lifted limit are well defined (thus making the lifted limit well defined also).

\medskip
To prove the main theorem, one has to gauge the lifting in a certain way in order to cut off the infinite series and, consequently, obtain an embedding of algebras. In the next section we argue that such gauging always exists.

\section{Lifted limit as a Weyl algebra automorphism}

Let $\sigma$ be an arbitrary polynomial symplectomorphism as before, and let $\psi$ be the lifted limit of a tame symplectomorphism sequence $\lbrace\sigma_k\rbrace$. In this section we actively use the embedding of the Weyl algebra over $\mathbb{C}$ into the reduced direct product of Weyl algebras over algebraically closed fields of positive characteristic $p$, which runs over all prime numbers. To avoid the conflict of notation, we denote the generators of $P_n$ and $W_n$ by letters $x$ and $y$ rather than (more classical) $x$ and $p$.

\medskip
We first observe the following
\begin{prop}
The power series $\Psi_l(x_1,\ldots,p_n,\hbar)$ which make up the lifted limit $\psi$ correspond to rational functions in $\hbar$. Namely, for each positive characteristic $p$ in the ultraproduct decomposition, the central elements $\Psi_l^p$ are rational functions in $\hbar$.
\end{prop}
\begin{proof}
Indeed, for a given positive characteristic $p$ marking a component in the ultraproduct, the correspoding Weyl algebra is Azumaya. This in particular means that it is isomorphic to the algebra of $p\times p$ matrices over its center. The two matrix algebras -- the initial Weyl algebra and the one which results from $\sigma$ -- are then isomorphic to each other, which is equivalent to the fact the the centers $C$ and $C_{\sigma}$ of the corresponding algebras are isomorphic\footnote{The fact that Morita equivalence of commutative rings implies their isomorphism is well known.}. Therefore, $\Psi_l^p$ are rational functions.
\end{proof}

We turn now to the proof of the main theorem. Our objective is to show that our lifting can be appropriately modified so that the resulting object will be given by polynomials. Working for each $p$ in the ultraproduct decomposition, we look for a gauge of the lifting that will leave the $\hbar$-independent part of the center, given by $\mathbb{F}_p[x_1^p,\ldots,y_n^p]$, unperturbed. We have the following

\begin{prop}
The lifting can be gauged in such a way that stabilizes the $\hbar$-independent center $\mathbb{F}_p[x_1^p,\ldots,y_n^p]$.
\end{prop}

Another important lemma is the following.

\begin{lem}
There is a transformation of the lifting that results in the lifted automorphism $\Psi$ being defined by power series in $\hbar$ which have coefficients polynomial in $x_i$ and $y_j$.
\end{lem}
The proof of this statement will be addressed in our further work. 

\smallskip

We next observe that the gauging must obviously be performed by means of an Azumaya algebra automorphism. In terms of matrix algebra, that means that (for fixed $p$) the desired transform must be a matrix conjugation. For $l=1,\ldots, 2n$, we therefore consider the expressions of the form
$$
(1+\hbar Q_l)\Psi_l(1+\hbar Q_l)^{-1},
$$
where $Q_l$ is a rational function in $\hbar$. Note that $Q_l$ is, generally, not an element of the Weyl algebra over $\mathbb{C}$ (in the standard sense), but rather a class (modulo ultrafilter $\mathcal{U}$ that realizes the ultraproduct decomposition) of elements $Q_{l,p}$ which are rational in $\hbar$.

\smallskip

\smallskip

Any such $Q_l$ leads to a gauging. 

We use the following result.

\begin{lem}
The polynomials $(\Psi_l)_1(x_1,\ldots, y_n)$, $l=1,\ldots, 2n$ which give the coefficients of $\hbar$-power series $\Psi_l$ at $\hbar$ are of total degree less than $\Deg \sigma$. 
\end{lem} 

This is not immediately obvious in light of remarks on the modified star product $\star_{\sigma}$ (although it is apparent for lifted tame symplectomorphisms). This statement provides the base case for the processing of the higher-order terms. 

\smallskip

The gauging acts upon the term at $\hbar^2$. Indeed, the leading term of the difference between the gauged and ungauged expression is given by the commutator (in the initial star product)
$$
[\Psi_l, Q_l]\hbar,
$$
which is of order $\hbar^2$ by definition of $\Psi_l$ and $Q_l$.

For $\sigma$ that are close enough to the identity symplectomorphism, this translates into correction terms of the form
$$
\frac{\partial Q_l}{\partial x_i}\;\;\text{and}\;\;\frac{\partial Q_l}{\partial y_j}.
$$
This can always be fulfilled.

Now, in order to tweak the higher-degree (in $\hbar$) terms, one applies consecutive conjugations according to the method described above. For the terms of degree higher than $\Deg \sigma$, the existence of compensating terms amounts to the vanishing of the appropriate differential form.

\medskip

We illustrate the algorithm by applying it to the base case $n=1$, so that the coefficients that need to be processed are power series in two generators $x$ and $y$. Firstly, we can always find a conjugation such that the resulting lifted limit will send $x$ to itself. Once this is done, the power series
$$
1+\hbar Q
$$
can only depend on $x$ ($Q$ corresponds to $y$). Conjugating the image of $y$ under $\Psi$ by $1+\hbar Q$, we can dispose of the terms that do not contain $y$. Indeed, the leading term of conjugation with $y$ is given by differentiation, so the amending term is constructed by taking a primitive of the appropriate polynomial.

On the other hand, because of the fact that the commutator $[y,x]=\hbar$ produces a power of $\hbar$ and shifts the terms one notch, no extra terms in the fixed term $\hbar^k$ (with which we are currently working) can appear. This completes the proof. The algorithm can be easily modified for the case of arbitrary $n$.

\section{Acknowledgments}
The work is supported by the Russian Science Foundation grant No. 17-11-01377.

\end{document}